\documentclass[11pt,a4paper,reqno]{amsart}
\usepackage{amsfonts}
\usepackage{amsthm}
\usepackage{amsmath}
\usepackage{amscd}
\usepackage[utf8]{inputenc}
\usepackage[T1]{fontenc} 
\usepackage[mathscr]{eucal}
\usepackage{indentfirst}
\usepackage{graphicx,color}
\usepackage{graphics}
\usepackage{pict2e}
\usepackage{epic}
\usepackage{epstopdf}

\usepackage{hyperref}

\theoremstyle{plain}
\newtheorem{Th}{Theorem}[section]
\newtheorem{Lemma}[Th]{Lemma}
\newtheorem{Cor}[Th]{Corollary}
\newtheorem{Prop}[Th]{Proposition}

\theoremstyle{definition}

\newtheorem{?}[Th]{Problem}

\def\bS{{\mathbb S}}
\def\SM{{\mathbb{S} M}}
\newcommand{\sm}{\bS\!\!\!/\,\!}
\newcommand{\D}{D\!\!\!\!/\,}
\newcommand{\nb}{\nabla\!\!\!\!/\,}
\newcommand{\mult}{\gamma\!\!\!/}
\newcommand{\E}{\mathcal{E}\!\!\!/}

\begin{document}
	
	\title[Spin hypersurfaces with constant scalar curvature]{A spinorial approach to constant scalar curvature hypersurfaces in pseudo-hyperbolic manifolds}

	\author[F. Gir\~{a}o and D. Rodrigues]{Frederico Gir\~{a}o\quad Diego Rodrigues }
	
	\address{Universidade Federal do Cear\'{a}\\ Departamento de Matem\'{a}tica\\Campus do Pici\\Av. Humberto Monte, s/n, bloco 914, 60455-760\\Fortaleza/CE\\Brazil} 
	
	\email{fred@mat.ufc.br}
	\email{diego.sousa.ismart@gmail.com}

	\begin{abstract} Using spinorial techniques, we prove, for a class of
	pseudo-hyperbolic ambient manifolds, a Heintze-Karcher type inequality. We then use this inequality to show an Alexandrov type theorem in such spaces.
	\end{abstract}
	
	\maketitle
	
	\section{Introduction}
	 The problem of classifying constant scalar curvature compact hypersurfaces in Euclidean space was proposed by Yau in  the problem section of \cite{yau}. One would like to know, for example, if Alexandrov's theorem \cite{aleksandrov}, which states that spheres are the only compact constant mean curvature hypersurfaces embedded in Euclidean space, still holds when the mean curvature is replaced by the scalar curvature. This question was answered positively by Ros in \cite{ros}, by modifying Reilly's proof of Alexandrov's theorem \cite{reilly}.
	After this, different proofs and generalizations of Ros'es result were given; for example, using the Minkowski formulae \cite{minkowski} and the Heintze-Karcher inequality \cite{HK}, Ros extended his result for any $r$-cur\-va\-ture in \cite{ros2}. This approach also works in hyperbolic space, where the same conclusion holds \cite{montiel2}. 
	
	
	Recall that a manifold of the form $M^{n+1}=\mathbb{R}\times_{\rm exp}P^n$, with $P$ a complete Riemannian manifold, is called a {\it pseudo-hyperbolic space} \cite{tashiro}. When $P^n$ is Ricci-flat, 
	$M^{n+1}$
	is Einstein with negative Ricci curvature, and when $P^n$ is flat, $M^{n+1}$ is a hyperbolic space form.
	
	
	In \cite{montiel}, among other results, Montiel proves the following Alexandrov type theorem: if $\Sigma$ is either a constant mean curvature or a constant scalar curvature hypersurface bounding a domain into a pseudo-hyperbolic space $\mathbb{R}\times_{\rm exp}P^n$, with $n \geq 2$ and $P^n$ a compact Ricci flat manifold, then it is either a geodesic sphere or a slice $\lbrace s \rbrace \times P^n$, $s \in \mathbb{R}$.

	
	The aim of this paper is to give a spinorial proof for the previous result. For this, we will prove a Heintze-Karcher inequality for spin manifolds carrying a nontrivial imaginary Killing spinor:
	
	\begin{Th}\label{main} Let $(M,g)$ be a (n+1)-dimensional connected Riemannian spin manifold carrying a nontrivial imaginary Killing spinor $\psi$ and let $\Sigma$ be a hypersurface bounding a compact domain $\Omega$ in $M$. Let $V=|\psi|^2$ and suppose the mean curvature $H$ of the hypersurface $\Sigma$ is positive everywhere. Then
		\begin{equation}\label{ineq1}
		\int_\Sigma\frac{V}{H}\,d\Sigma+\int_\Sigma\langle\nabla V,N\rangle\,d\Sigma\geq0,
		\end{equation}
	where $\nabla$ is the Levi-Civita connection of $(M,g)$ and $N$ is the inward pointing unit vector field normal to $\Sigma$. Moreover, equality holds if and only if $\Sigma$ is totally umbilical.
	\end{Th}
	
	Let $(M,g)$ be as in Theorem \ref{main}, that is, $(M,g)$ is a Riemannian spin manifold carrying a nontrivial imaginary Killing spinor. When $M$ is complete, Baum proved in \cite{baum} that $M$ is a warped product $\mathbb{R}\times_{\rm exp}P$, with the $n$-dimensional manifold $P$ being a complete Riemannian spin manifold admitting a nontrivial parallel spinor. Hence, by Wang's classification \cite{wang}, $P$ is a {\bf flat manifold, a Calabi-Yau manifold, a hyper-K\"{a}hler manifold or some eight- or seven-dimensional Riemanian manifolds with special holonomy}. Also, as we will see later, the function $V$ satisfies
	$$
	\Delta V = (n+1)V.
	$$
	Thus, Theorem \ref{main} can be rewritten as follows:
	
	\begin{Cor}\label{main2} Let $\Sigma$ be a connected compact hypersurface bounding a compact domain in a pseudo-hyperbolic space $M=\mathbb{R}\times_{\rm exp}P$, where $P$ is a complete Riemannian spin manifold admitting a nontrivial parallel spinor. Assume that the mean curvature $H$ of the hypersurface $\Sigma$ is positive everywhere. Then
		\begin{equation}\label{ineq2}
		\int_\Sigma\frac{V}{H}\,d\Sigma\geq (n+1)\int_\Omega V\,dvol.
		\end{equation}
		Moreover, equality holds if and only if $\Sigma$ is totally umbilical.
	\end{Cor}
	
	An interesting special case of Corollary \ref{main2} is when $P$ is the Euclidean space $\mathbb{R}^n$, which implies that $M$ is the hyperbolic space $\mathbb{H}^n$. Note that, in this case, the conclusion ``$\Sigma$ is totally umbilical'' in the equality case can be changed to ``$\Sigma$ is a geodesic sphere''. This case follows from a very general result of Brendle for warped product ambient manifolds \cite{brendle} (see also \cite{Qiu-Xia,ww}). A spinorial proof of this special case was given by Hijazi, Montiel and Raulot in \cite{hss}.
	They acomplish this by realizing $\mathbb{H}^{n+1}$ as a spacelike hypersurface in Minkowski space $\mathbb{R}^{n+1,1}$ and using spinorial techniques in such ambient. Our proof, in its turn, is totally implicit and it's valid for a large class of ambient spaces; one of its main ingredients is a holographic principle for the existence of imaginary killing spinors, also due to Hijazi, Montiel and Raulot \cite{hor} (see Theorem \ref{holographic}).

	
	As mentioned before, in \cite{montiel}, Montiel showed an Alexandrov type theorem for hypersurfaces with constant mean curvature or constant scalar curvature in some pseudo-hyperbolic spaces. Using spinorial techniques, Hijazi, Montiel and Roldan gave another proof to the constant mean curvature case \cite{hor2}. Here, we give another proof to the constant scalar curvature case; our proof is spinorial in the sense that it uses inequality (\ref{ineq1}), which was proved using spinorial techniques.

	
	\begin{Cor}\label{main3} Let $\Sigma$ be a connected hypersurface bounding a  domain in a pseudo-hyperbolic space $\mathbb{R}\times_{\rm exp}P$, where $P$ is a complete Riemannian spin manifold admitting a nontrivial parallel spinor. If  the scalar curvature of $\Sigma$ is constant, then it is either a round geodesic hypersphere (and, in this case P must be flat) or a slice $\{s\}\times P,\ s\in\mathbb{R}$.
	\end{Cor}
	\section{Preliminaires}
	In this section, we recall some definitions and properties of the spin geometry of hypersurfaces embedded in a spin manifold, as it is done in \cite{hor}. 
	
	Let $(M,\langle\ ,\ \rangle)$ be a $(n+1)$-dimensional Riemannian spin manifold. We fix a spin structure and let $\mathbb{S}M$ denote the corresponding spinor bundle. We denote by $\overline{\nabla}$ both the Levi-Civita connection of $(M,\langle\ ,\ \rangle)$ and its lift to $\mathbb{S}M$, and by $\overline{\gamma}:\mathbb{C}\ell(M)\to{\rm End}_\mathbb{C}(\mathbb{S}M)$ the Clifford multiplication.
	On the spinor bundle $\mathbb{S}M$ there exists a natural Hermitian structure (see \cite{lawson}) denoted, as the Riemannian metric on $M$, by $\langle\ ,\ \rangle$. The spinorial Levi-Civita connection and the Hermitian product are compatible with the Clifford multiplication and compatible with each other; that is,
	for any $X,Y\in\Gamma(TM)$ and any $\psi,\varphi\in\Gamma(\mathbb{S}M)$ the following identities hold:
	\begin{align}
	X\langle\psi,\varphi\rangle&=\langle\overline{\nabla}_X\psi,\varphi\rangle+\langle\psi,\overline{\nabla}_X\varphi\rangle\label{c1};\\
	\langle\overline{\gamma}(X)\psi,\varphi\rangle&=-\langle\psi,\overline{\gamma}(X)\varphi\rangle\label{c2};\\
	\overline{\nabla}_X\left(\overline{\gamma}(Y)\psi\right)&=\overline{\gamma}(\overline{\nabla}_XY)\psi+\overline{\gamma}(Y)\overline{\nabla}_X\psi\label{c3}.
	\end{align}
	Also, the Dirac operator $\overline{D}$ on $\mathbb{S}M$ is locally given by
	\begin{equation}
	\overline{D}=\sum_{i=1}^{n+1}\overline{\gamma}(e_i)\overline{\nabla}_{e_i},
	\end{equation}
	where $\{e_1,\ldots,e_{n+1}\}$ is a local orthonormal frame of $TM$.
	
	Consider an orientable hypersurface $\Sigma$ immersed into $M$. The Riemannian metric on $M$ induces a Riemannian metric on $\Sigma$, also denoted by $\langle\ ,\ \rangle$, whose Levi-Civita connection $\nabla^\Sigma$ satifies the Riemannian Gauss formula
	\begin{equation}\label{gauss}
	\nabla^\Sigma_XY=\overline{\nabla}_XY-\langle A(X),Y\rangle N,
	\end{equation}
	where $X,Y$ are vector fields tangent to the hypersurface $\Sigma$, the vector field $N$ 
	is the inward pointing unit vector field normal to 
	$\Sigma$, and $A$ is the shape operator with respect to $N$, that is,
	$$\overline{\nabla}_XN=-AX,\quad\forall X\in\Gamma(T\Sigma).$$	
	
	Since the normal bundle of $\Sigma$ is trivial, the hypersurface $\Sigma$ inherits a spin structure from the one of the ambient manifold $M$. Thus, $\Sigma$ has a Hermitian spinor bundle $\mathbb{S}\Sigma$ in the sense of \cite{lawson}, i.e., a \textit{Dirac bundle}. We will denote by $\gamma^\Sigma$ and $D^\Sigma$, respectively, the Clifford multiplication and the intrisic Dirac operator on $\Sigma$. We call such spinor bundle by \textit{intrisic} spinor bundle. We compare the intrisic spinor bundle $\mathbb{S}\Sigma$ to the restriction $\sm\Sigma=\SM_{|\Sigma}$. This bundle is isomorphic to either $\bS\Sigma$ or $\bS\Sigma\oplus\bS\Sigma$ acording to the dimension $n$ of $\Sigma$ is either even or odd (\cite{bar,morel}). Since the $n$-dimensional Clifford algebra is the even part of the $(n+1)$-dimensional Clifford algebra, the Clifford multiplication $\mult:\mathbb{C}\ell(\Sigma)\to{\rm End}_\mathbb{C}(\sm\Sigma)$ is given by
	\begin{equation}\label{cliff}
	\mult(X)\psi=\overline{\gamma}(X)\overline{\gamma}(N)\psi,
	\end{equation}
	for any $\psi\in\Gamma(\bS\Sigma)$ and any $X\in\Gamma(T\Sigma)$. Consider on $\sm\Sigma$ the Hermitian metric $\langle\ ,\ \rangle$ induced from that of $\SM$. This metric satisfies te compatibilty condition (\ref{c2}) if one considers on $\Sigma$ the Riemannian metric induced from $M$ and the Clifford  multicplication $\mult$ defined by (\ref{cliff}). The Gauss formula (\ref{gauss}) implies that the spin connection $\nb$ on $\sm\Sigma$ is given by the following spinorial Gauss formula
	\begin{equation}\label{nb}
	\nb_X\psi=\overline{\nabla}_X\psi-\frac{1}{2}\mult(AX)\psi
	\end{equation}
	for any $\psi\in\Gamma(\sm\Sigma)$ and any $X\in\Gamma(T\Sigma)$. Observe that the compatibity conditions (\ref{c1}), (\ref{c2}) and (\ref{c3}) are satisfied by $(\sm\Sigma,\mult,\langle\ ,\ \rangle,\nb)$.
	
	The \textit{extrinsic} Dirac operator $\D=\mult\circ\nb$ on $\Sigma$ defines a first order elliptic operator acting on sections of $\bS\Sigma$. By (\ref{nb}), for any spinor field $\psi\in\Gamma(\bS\Sigma)$ we have
	
	\begin{equation}\label{extdirac}
	\D\psi=\sum_{i=1}^n\mult(e_i)\nb_{e_i}\psi=\frac{n}{2}H\psi-\overline{\gamma}(N)\sum_{i=1}^n\overline{\gamma}(e_i)\overline{\nabla}_{e_i}\psi,
	\end{equation}
	and
	\begin{equation}
	\D(\overline{\gamma}(N)\psi)=-\overline{\gamma}(N)\D\psi,
	\end{equation}
	where $\{e_1,\ldots,e_n\}$ is a local orthonormal frame of $T\Sigma$ and $H=\frac{1}{n}\textrm{trace}A$ is the mean curvature of $\Sigma$ in $M$. Thus, we have on $\Sigma$ an intrinsic spinorial structure $(\bS\Sigma,\nabla^\Sigma,\gamma^\Sigma,D^\Sigma)$ and an extrinsic structure $(\sm\Sigma,\nb,\mult,\D)$. The dimension of $\Sigma$ plays an important role in the isomorphism of such structures. In fact, if $n$ is even, then
	\begin{equation}\label{restrict_even}
     (\sm\Sigma,\nb,\mult,\D)\equiv(\bS\Sigma,\nabla^\Sigma,\gamma^\Sigma,D^\Sigma)
	\end{equation}
and, if $n$ is odd, then
\begin{equation}\label{restrict_odd}
(\sm\Sigma,\nb,\mult,\D)\equiv(\bS\Sigma\oplus\bS\Sigma,\nabla^\Sigma\oplus\nabla^\Sigma,\gamma^\Sigma\oplus-\gamma^\Sigma,D^\Sigma\oplus-D^\Sigma).
\end{equation}	

Now, we recall the definition of a \textit{chirality operator}. A chiralty operator $\omega$ on a Dirac bundle $(\mathcal{E}M,\gamma,\nabla,\langle\ ,\ \rangle)$ is an endomorphism $\omega:\Gamma(\mathcal{E}M)\to\Gamma(\mathcal{E}M)$ such that
\begin{align}
    \omega^2=\textrm{Id}_{\mathcal{E}M},&\qquad \langle\omega\psi,\omega\varphi\rangle=\langle\psi,\varphi\rangle,\label{chi1}\\
    \omega(\gamma(X)\psi)=-\gamma(X)\omega\psi,&\qquad\nabla_X(\omega\psi)=\omega(\nabla_X\psi),\label{chi2}
\end{align}
for any $X\in\Gamma(TM)$ and any $\psi,\varphi\in\Gamma(\mathcal{E}M)$.	

Now, we set up a new Dirac bundle with a chirality operator. Consider the vector bundle
$$\mathcal{E}M:=\left\{
\begin{array}{ll}
\SM&\textrm{ if $n+1$ is even},\\
\SM\oplus\SM&\textrm{ if $n+1$ is odd,}
\end{array}
\right.$$
equipped with a Clifford multiplication $\gamma$	defined by
$$\gamma=\left\{
\begin{array}{ll}
\overline{\gamma}&\textrm{ if $n+1$ is even},\\
\overline{\gamma}\oplus-\overline{\gamma}&\textrm{ if $n+1$ is odd,}
\end{array}
\right.$$	
and a Levi-Civita connection
$$\nabla=\left\{
\begin{array}{ll}
\overline{\nabla}&\textrm{ if $n+1$ is even},\\
\overline{\nabla}\oplus\overline{\nabla}&\textrm{ if $n+1$ is odd}.
\end{array}
\right.$$	
Futhermore, $\langle\ ,\ \rangle$ denotes the Hermitian scalar product given by $\langle\ ,\ \rangle_M$ for $n$ odd and by
$$\langle\Psi,\Phi\rangle:=\langle\psi_1,\varphi_1\rangle_M+\langle\psi_2,\varphi_2\rangle_M$$
for $n$ even, for any $\Psi=(\psi_1,\psi_2)$, $\Phi=(\varphi_1,\varphi_2)\in\Gamma(\mathcal{E}M)$.

It is straightforward to verify that $(\mathcal{E}M,\nabla,\gamma)$ is a Dirac bundle in the sense of \cite{lawson}. The Dirac-type operator acting on sections of $\mathcal{E}M$ and defined by $D:=\gamma\circ\nabla$ is explicity given by
$$D=\left\{
\begin{array}{ll}
     \overline{D}&\textrm{ if $n+1$ is even,}  \\
     \overline{D}\oplus-\overline{D}&\textrm{ if $n+1$ is odd.} 
\end{array}
\right.$$

As it is done in \cite{hor}, let us examine this bundle and its restriction $(\E,\nb,\mult)$ to $\Sigma$.

If $n+1$ is even, the operator $\omega:=\gamma(\omega_{n+1}^\mathbb{C})$ defines a chirality operator on $\SM$, where $\omega_{n+1}^\mathbb{C}=i^{\left[\frac{n+2}{2}\right]}e_1\cdot\ldots\cdot e_{n+1}$ is the complex volume element. Moreover, the spinor bundle splits into
$$\mathcal{E}M=\SM=\bS^+M\oplus\bS^-M,$$
where $\bS^\pm M$ are the   $\pm1$-eigenspaces of the endomorphism $\omega$. On the other hand, the restricted spinor bundle
$$\E:=\mathcal{E}M_{|\Sigma}=\SM_{|\Sigma}=\sm\Sigma$$
can be identified  with the intrinsic data of $\Sigma$ as in (\ref{restrict_odd}).

If $(n+1)$ is odd, $\mathcal{E}M=\SM\oplus\SM$ and the map
$$
\begin{array}{rcl}
\omega:\Gamma(\mathcal{E}M)&\longrightarrow&\Gamma(\mathcal{E}M)\\
\begin{pmatrix}
\psi_1\\\psi_2
\end{pmatrix}&\longmapsto&\begin{pmatrix}
\psi_2\\\psi_1
\end{pmatrix},
\end{array}
$$
satisfies the properties (\ref{chi1}) and (\ref{chi2}), so that it defines a chirality operator on $\mathcal{E}M$. The restriction of $\mathcal{E}M$ to $\Sigma$ is given by
$$\E:=\mathcal{E}M_{|\Sigma}=\sm\Sigma\oplus\sm\Sigma$$
and can be identified with to copies of the intrisic spinor bundle of $\Sigma$ as in (\ref{restrict_even}).

The extrinsic Dirac operator acting on sections of $\E$ is defined by $\D:=\mult\circ\nb$.
We define the modified Dirac-type operators on $\mathcal{E}M$ and $\E$, respectively, by
\begin{equation}\label{modified_dirac_M}
    D^{\pm}:=D\mp\frac{n+1}{2}i\textrm{ Id}_{\mathcal{E}M}
\end{equation}
and
\begin{equation}\label{dirac+-}
    \D^\pm:=\D\pm\frac{n}{2}i\gamma(N)\textrm{ Id}_{\E}.
\end{equation}
If $M$ admits a Killing imaginary spinor field $\psi_\pm\in\Gamma(\mathcal{E}M)$ with Killing number $\pm\frac{i}{2}$ , i.e., $$\nabla_X\psi_{\pm}=\pm\frac{i}{2}\gamma(X)\psi_\pm,$$
for any $X\in\Gamma(TM),$ we can show that
\begin{equation}
    D^\mp\psi_\pm=0     \quad\textrm{and}\quad \D^\mp\psi_\pm=\frac{nH}{2}\psi_\pm.
\end{equation}
The previous discussion can be summarized in the following proposition (see \cite{hor}):
\begin{Prop}[\cite{hor}] The bundle $(\mathcal{E}M,\gamma,\nabla)$ is a Dirac bundle equipped with a chirality operator $\omega$ whose associated Dirac-type operator $D:=\gamma\circ\nabla$ is a first order elliptic differential operator. The restricted triplet $(\E,\mult,\nb)$ is also a Dirac bundle for which the spinorial Gauss formula
\begin{equation}
    \nb_X\psi=\nabla_X\psi-\frac{1}{2}\mult(AX)\psi
\end{equation}
holds for all $\psi\in\Gamma(\E)$ and $X\in\Gamma(T\Sigma)$, and such that
\begin{equation}
    \D\psi=\frac{n}{2}H\psi-\gamma(N)D\psi-\nabla_N\psi
\end{equation}
and
\begin{equation}
    \D(\gamma(N)\psi)=-\gamma(N)\D\psi,
\end{equation}
where $\D:=\mult\circ\nb$ is the extrinsic Dirac-type operator on $\E$. Moreover, the Dirac-type operators $$\D^\pm:=\D\pm\frac{n}{2}i\gamma(N)\textrm{Id}_{\E}$$ are first order differential operators which only depend on the Riemannian and spin structures of $\Sigma$.
\end{Prop}
Now, consider the operator
$$G:=\gamma(N)\omega: \Gamma(\E)\to\Gamma(\E).$$ 
This endomorphism is a self-ajoint involution with respect to the pointwise Hermitian scalar product $\langle\ ,\ \rangle$, where $\omega$ is the chirality operator on $\mathcal{E}M$. 
It induces an orthogonal decomposition of $\E$:
\begin{equation}
    \E=\mathcal{V}^+\oplus\mathcal{V}^-,
\end{equation}
where $\mathcal{V}^\pm$ are eigeinsubbundles over $\Sigma$ corresponding to the $\pm1$-eigenvalues of $G$. Thus, we define the associated projections on $\mathcal{V}^\pm$:
\begin{equation}
\begin{array}{rcl}
P_\pm:L^2(\mathcal{E}M)&\longrightarrow&L^2(\mathcal{V}^\pm)\\
\psi&\longmapsto&P_\pm\psi:=\frac{1}{2}(\textrm{Id}_{\mathcal{E}M}\pm\gamma(N)\omega)\psi,
\end{array}
\end{equation}
where $L^2(\mathcal{E}M)$ and $L^2(\mathcal{V}^\pm)$ denote, respectively, the spaces of $L^2$-integrable sections of $\mathcal{E}M$ and $\mathcal{V}^\pm$. The projections $P_\pm$ are orthogonal to each other and are self-adjoint with respect to the pointwise Hermitian scalar product $\langle\ ,\  \rangle$. Also, we can check that
\begin{equation}
    \D^+P_\pm=P_\mp\D^+.
\end{equation}

We end this section by stating the following result, due to Hijazi, Montiel and Raulot, which will be a key ingredient in the proof of Theorem \ref{main}.
	\begin{Th} \label{holographic} Let $\Omega$ be a compact, connected Riemannian spin manifold with smooth boundary $\Sigma$. Assume that the scalar curvature of $\Omega$ satisfies $R\geq -n(n+1)k^2$ for some $k>0$ and the mean curvature $H$ of $\Sigma$ is positive. Then for all $\Phi\in\Gamma(\E)$, one has
		\begin{equation}\label{ineq}
		\int_\Sigma\left(\frac{1}{H}|\D^+\Phi|^2-\frac{n^2}{4}H|\Phi|^2\right)\,d\Sigma\geq0.
		\end{equation}
		Moreovoer, equality occurs for $\Phi\in\Gamma(\E)$ if and only if there exists two imaginary Killing spinor fields $\Psi^+,\Psi^-\in\Gamma(\E)$ with Killing number $-(i/2)$ such that $\mathcal{P}_+\Psi^+=\mathcal{P}_+\Phi$ and $\mathcal{P}_-\Psi^-=\mathcal{P}_-\Phi$.
	\end{Th}

\section{Proof of Theorem~\ref{main}}
	
	In this section we present the proof of Theorem~\ref{main}.
	
	\begin{proof} Assume $H>0$ on $\Sigma$. Let $\psi$ a imaginary Killing spinor field with Killing number $i/2$ on $\Omega$, such that $V=|\psi|^2$ (See \cite{baum}).  We take the spinor field $\varphi=\psi|_\Sigma$ on $\Sigma$; for such $\varphi$, we have
		\begin{eqnarray*}
			\D\varphi&=&\frac{nH}{2}\psi-\gamma(N)\sum_{i=1}^n\gamma(e_i)\nabla_{e_i}\psi\\
			&=&\frac{nH}{2}\varphi+\frac{in}{2}\gamma(N)\psi,
		\end{eqnarray*}
		and
		\begin{equation*}
		\D^+\varphi=\frac{nH}{2}\varphi+in\gamma(N)\psi.
		\end{equation*}
		Thus,
		$$|\D^+\varphi|^2=\frac{n^2H^2}{4}|\varphi|^2+n^2|\psi|^2+n^2H\Re\langle i\gamma(N)\psi,\psi\rangle.$$
		Hence, we get
		$$\int_\Sigma\frac{1}{H}|\D^+\varphi|^2\,d\Sigma=\int_\Sigma\frac{n^2H}{4}|\varphi|^2\,d\Sigma+n^2\left(\int_\Sigma\frac{|\psi|^2}{H}\,d\Sigma+\int_\Sigma\Re\langle i\gamma(N)\psi,\psi\rangle\,d\Sigma\right).$$
		Now, we apply (\ref{ineq}) to obtain
		\begin{equation}\label{ineqspin}
		\int_\Sigma\frac{|\psi|^2}{H}\,d\Sigma+\int_\Sigma\Re\langle i\gamma(N)\psi,\psi\rangle\,d\Sigma\geq0.
		\end{equation}
		By other hand, we have that
			\begin{equation}\label{lemma}
			\langle\nabla V,N\rangle=\Re\langle i\gamma(N)\psi,\psi\rangle.
			\end{equation}
			Indeed, using the fact $\nabla_X\psi=\frac{i}{2}\gamma(X)\psi$, for all $X\in\Gamma(TM)$, we get
			\begin{eqnarray*}
				\langle\nabla V,N\rangle&=&N|\psi|^2\\
				&=&\langle\nabla_N\psi,\psi\rangle+\langle\psi,\nabla_N\psi\rangle\\
				&=&2\Re\langle\nabla_N\psi,\psi\rangle\\
				&=&\Re\langle i\gamma(N)\psi,\psi\rangle.
			\end{eqnarray*}

		Thus, substituing (\ref{lemma}) in (\ref{ineqspin}) and remembering that $V=|\psi|^2$, we obtain
		$$\int_\Sigma\frac{V}{H}\,d\Sigma+\int_\Sigma\langle\nabla V,N\rangle\,d\Sigma\geq0.$$
		
		The equality holds if and only if we have equality in (\ref{ineq}). In such case, there exists two imaginary Killing spinor fields $\Psi^+,\Psi^-\in\Gamma(\E)$ with Killing number $-(i/2)$ such that $\mathcal{P}_+\Psi^+=\mathcal{P}_+\varphi$ and $\mathcal{P}_-\Psi^-=\mathcal{P}_-\varphi$ on $\Sigma$. Then, $\varphi=\mathcal{P}_+\Psi^++\mathcal{P}_{-}\Psi^-$, then
		\begin{eqnarray*}
			\D^+\varphi&=&\D^+(\mathcal{P}_+\Psi^+)+\D^+(\mathcal{P}_{-}\Psi^-)\\
			&=&\mathcal{P}_{-}(\D^+\Psi^+)+\mathcal{P}_+(\D^+\Psi^-)\\
			&=&\frac{nH}{2}(\mathcal{P}_{-}\Psi^++\mathcal{P}_+\Psi^-).
		\end{eqnarray*}
		We deduce that
		$$\frac{2}{nH}\D^+\varphi+\varphi=\Psi^++\Psi^-=\widetilde{\Psi},$$
		that is
		$$\frac{2}{nH}\left(\frac{nH}{2}\psi+in\gamma(N)\psi\right)+\psi=\widetilde{\Psi},$$
		then
		$$\psi+\frac{i\gamma(N)}{H}\psi =\frac{1}{2}\widetilde{\Psi}.$$
		The spinor field $\widetilde{\Psi}$ is imaginary Killing since $\Psi^+$ and $\Psi^-$ are, moreover $\widetilde{\Psi}$ has Killing number $-i/2$, so the spinor field
		$$\psi+\frac{i\gamma(N)}{H}\psi$$
		is a restrition of a imaginary Killing spinor field with Killing number $-i/2$, therefore, for all $X\in\Gamma(T\Sigma)$:
		$$H\gamma(X)\psi-\gamma(AX)\psi-\frac{1}{H}X(H)\gamma(N)\psi=0.$$
		Now we choose $X=X_i\in\Gamma(\Sigma)$, where $X_i$ is a direction of principal curvatures of $\Sigma$, whose associated principal curvature is $\lambda_i$. Taking the scalar product of the last equality  with $\gamma(X_i)\psi$, we get
		$$H|X_i|^2|\psi|^2-\lambda_i|X_i|^2|\psi|^2=0.$$
		Since $|\psi|^2=V\geq1$ and at each point $p\in\Sigma$ we can choose a basis $(X_1,\ldots,X_n)$ of $T_p\Sigma$ such that $X_i$
		is a direction of principal curvature, we get $\lambda_i=H$ on $\Sigma$ for all $i\in\{1,\ldots,n\}$, so 
		$$A=H\,{\rm Id}.$$
		Thus, $\Sigma$ is totally umbilical. 
	\end{proof}
	
	\section{Proof of Corollary~\ref{main2}}
	\begin{proof} 
		Let $\psi$ be an imaginary Killing spinor with Killing number $i/2$ (after a rescaling of the metric). Thus, for each $X\in\Gamma(TM)$ we have
		\begin{equation}\label{killing}
		\nabla_X\psi=\frac{i}{2}\gamma(X)\psi.
		\end{equation}
		Setting $V=|\psi|^2$, one can check from (\ref{killing}) that $V$ satisfies
		\begin{equation}\label{hess}
		{\rm Hess}\,V=V\langle\ ,\ \rangle.
		\end{equation}
		Thus, tracing (\ref{hess}), follows
		
		$$\Delta V=(n+1)V.$$
		Integrating that equation on the compact domain $\Omega$ and applying the divergence theorem, from (\ref{ineq1}) we obtain (\ref{ineq2}).
		
		Now, if the equality holds in (\ref{ineq2}), by Theorem \ref{main}, $\Sigma$ must be totally umbilical.
		
		In particular, when $P=\mathbb{R}^n$, the manifold $M$ is isometric to the hyperbolic space $\mathbb{H}^{n+1}$. Thus, $\Sigma$ is a totally umbilical hypersurface of $\mathbb{H}^{n+1}$ and so it is a geodesic sphere.
		
	\end{proof}
	
	\section{proof of corollary \ref{main3}}
	We begin this section remembering to the reader some facts about the geometry of hypersurfaces in Riemannian manifolds. On a given hypersurface $\Sigma$ in $M$, we define the $k$-th mean curvature function
	$$H_k=H_k(\Lambda)=\frac{1}{\binom{n}{k}}\sigma_k(\Lambda),$$
	where $\Lambda=(\lambda_1,\cdots,\lambda_n)$ are the principal curvature functions on $\Sigma$ and the homogeneous polynomial $\sigma_k$ of degree $k$ is the $k$-th elementary symmetric function
	$$\sigma_k(\Lambda)=\sum_{i_1<\cdots <i_k}\lambda_{i_1}\cdots\lambda_{i_k}.$$
	The next proposition gives an relation between these curvatures:
	\begin{Prop}(See \cite{garding,montiel2})
		Let $x:\Sigma\to M$ an isometric immersion between two Riemannian manifolds of dimension $n$ and $(n+1)$ respectively, and assume $\Sigma$ is connected. We suppose that there is a point of $\Sigma$ where all principal curvatures are positive. Then, if there exists $k\in\{1,\ldots, n\}$ such that $H_k>0$ on $\Sigma$, then
		\begin{equation}\label{garding}
		H\geq H_2^{1/2}\geq\cdots\geq H_r^{1/r}\quad{\rm on}\ \Sigma.
		\end{equation}
		If $k\geq 2$, equality holds only at umbilical points.
	\end{Prop}
	
	Now, if $\nabla$ denotes the Levi-Civita connection on $M$ and $N$ the unit normal vector field along $\Sigma$ which points to the inner region, we define the shape operator $A$ by $A(X)=-\nabla_XN$.
	Thus, the classical Newton transformations $T_k:\Gamma(T\Sigma)\to\Gamma(T\Sigma)$ are defined inductively from $A$ by:
	$$T_0=I,\quad{\rm and}\quad T_k=\sigma_kI-AT_{k-1},\quad 1\leq k\leq n,$$
	where $I$ denotes the identity in $\Gamma(T\Sigma)$.
	
	Associated to each Newton transformation $T_k$ one has the second order linear differential operator $L_k:\mathcal{C}^\infty(\Sigma)\to\mathcal{C}^\infty(\Sigma)$ for $k=0,1,\ldots,n-1$, given by
	$$L_k(u)=\textrm{tr}(T_k\circ\textrm{Hess }u),$$
	where ${\rm Hess} u:\Gamma(T\Sigma)\to\Gamma(T\Sigma)$ denotes the symmetric operator defined by
	$${\rm Hess} u(X)=\nabla^\Sigma_X\nabla^\Sigma u,\quad \forall X\in\Gamma(T\Sigma).$$
	In particular $L_0=\Delta$ is the Laplace-Beltrami operator. while $L_1$ is the operator $\square$, introduced by Cheng and Yau \cite{cheng-yau} for the study of hypersurfaces with constant scalar curvature.
	
	On the other hand, the divergent of $T_k$ is defined by
	$${\rm div}_\Sigma T_k=\sum_{i=1}^n\left(\nabla^{\Sigma}_{e_i}T_k\right)(e_i),$$
	where $\{e_1,\cdots,e_n\}$ is a local orthonormal frame on $\Sigma$. Thus, we have
	\begin{equation}\label{Lr}
	L_k(u)=\textrm{div}_\Sigma(T_k(\nabla^\Sigma u))-\langle\textrm{div}_\Sigma T_k,\nabla^\Sigma u\rangle.
	\end{equation}
	From (\ref{Lr}), we conclude that the operator $L_k$ is elliptic if, and only if, $T_k$ is positive definite. Clearly, $L_0=\Delta$ is always elliptic. The ellipticity of $L_1$ is guaranteed by Lemma 3.10 of \cite{elbert} when $H_2>0$.
	
	If the ambient space $M$ is equipped with a conformal vector field $Y\in\mathcal{X}(M)$, with conformal function $f$, then is shown in \cite{alias-lira} that
	\begin{equation}\label{div_Newton}
	{\rm div}_{\Sigma}(T_kY^{\top})=\langle{\rm div}_\Sigma T_k,Y\rangle+c_k\left(fH_k+\langle Y,N\rangle H_{k+1}\right),
	\end{equation}
	where 
	$$c_k=(k+1)\binom{n}{k+1}.$$
	Integrating (\ref{div_Newton}) over $\Sigma$ and making use of Divergence theorem we obtain
	\begin{equation}\label{mink}
	\int_\Sigma\langle{\rm div}_\Sigma T_k,Y\rangle\,d\Sigma+c_k\int_\Sigma\left(fH_k+\langle Y,N\rangle H_{k+1}\right)\,d\Sigma=0.
	\end{equation}
	
	A useful formula is obtained in \cite{alias-lira} for every tangent field $X\in\Gamma(T\Sigma)$:
	\begin{equation}\label{div_curv}
	\langle{\rm div}_\Sigma T_k,X\rangle=\sum_{j=1}^k\sum_{i=1}^n\langle R(N,T_{k-j}e_i)e_i,A^{j-1}X\rangle.
	\end{equation}
	In particular, when the ambient space $M$ has constant curvature, then $\langle R(N,V)W,Z\rangle=0$ for every tangent vector fields $V,W,Z\in\Gamma(T\Sigma)$, from (\ref{div_curv}) and (\ref{div_Newton}) we obtain the classical Minkowski integral identity for spaces with constant curvature:
	
	$$\int_\Sigma\left(fH_k+\langle Y,N\rangle H_{k+1}\right)\,d\Sigma=0.$$
	
	By other hand, when the ambient space is an Einstein manifold, taking (\ref{div_curv}) with $k=1$ we get
	$$\langle{\rm div}_\Sigma T_1,X\rangle=Ric(N,X)=0.$$
	Thus, for a compact hypersurface in a Einstein space the following is valid:
	\begin{equation}\label{mink2}
	\int_\Sigma fH_1\,d\Sigma+\int_\Sigma\langle Y,N\rangle H_{2}\,d\Sigma=0.
	\end{equation}
	
	Since every Riemannian spin manifold admitting a imaginary Killing vector is a Einstein manifold with Ricci curvature $-n$, and taking $Y=\nabla V$, where $V=|\psi|^2$, we have from (\ref{mink2}):
	\begin{equation}\label{mink3}
	\int_\Sigma VH_1\,d\Sigma+\int_\Sigma\langle \nabla V,N\rangle H_{2}\,d\Sigma=0.
	\end{equation}
	
	\begin{proof}[ Proof of Corollary \ref{main3}]
		The scalar curvature $S^{\Sigma}$ of a hypersurface can be related with the scalar curvature $S$ of the ambient space by the following formula:
		$$S^{\Sigma}=S-2Ric(N,N)+n(n-1)H_2.$$
		In our case, we have
		$$S^{\Sigma}=n(n-1)(H_2-1).$$
		Thus, the hypothesis of constant scalar curvature is equivalent to $H_2$ constant.
		
		Now, it is easy verify that, with respect to the normal $-\frac{\partial}{\partial t}$, the slices $\Sigma_s=\{s\}\times P$ are totally umbilical hypersurfaces with constant principal curvatures equal 1.
		
		Since $\Sigma$ is compact, there is an point $p\in\Sigma$ such that all principal curvatures of $\Sigma$ are bounded from below by 1 (this can be done by choosing a point $p$ where the projetion onto the real line is maximum), thus the constant $H_2$ is bounded from below by 1, and so is $H$ by (\ref{garding}). 
		
		First, consider the case where $\Sigma$ bounds a compact domain. For this case, the following lemma will be necessary:
		\begin{Lemma}\label{lemma2} If $H_2$ is constant, we have
			
			$$\int_\Sigma \left(\sqrt{H_2}-H\right)\langle\nabla V,N\rangle\,d\Sigma\leq0.$$
			Equality holds if and only if $\Sigma$ is totally umbilical.
		\end{Lemma}
		\begin{proof}
			From (\ref{garding}) we have
			$$\int_\Sigma VH  \,d\Sigma\geq\int_\Sigma   V\sqrt{H_2}\,d\Sigma=\sqrt{H_2}\int_\Sigma V\,d\Sigma.$$
			By (\ref{mink3}), we get
			$$-\int_\Sigma   H_2\langle\nabla V,N\rangle\,d\Sigma\geq\sqrt{H_2}\int_\Sigma V\,d\Sigma.$$
			Thus,
			$$-\int_\Sigma   H_2\langle\nabla V,N\rangle\,d\Sigma\geq\sqrt{H_2}\int_\Sigma \langle\nabla V,N\rangle H\,d\Sigma.$$
			Finally, we obtain
			$$\int_\Sigma \left(\sqrt{H_2}-H\right)\langle\nabla V,N\rangle\,d\Sigma\leq0,$$
			with equality if and only if $H=\sqrt{H_2}$ on $\Sigma$. This is equivalent to $\Sigma$ being totally umbilical.
		\end{proof}

		Let $\psi$ be an imaginary Killing spinor field with spinor number $i/2$ and define $\varphi:=(\sqrt{H_2}+i\gamma(N))\psi$ on $\Sigma$. First, by (\ref{ineq1}) and (\ref{lemma}), we get
		\begin{align*}
		\int_\Sigma\langle i\gamma(N)\psi,\varphi\rangle\,d\Sigma&=\sqrt{H_2}\int_\Sigma \langle\nabla V,N\rangle\,d\Sigma+\int_\Sigma V\,d\Sigma\\
		&\geq-\int_\Sigma\frac{\sqrt{H_2}}{H}V\,d\Sigma+\int_\Sigma V\,d\Sigma\\
		&=\int_\Sigma\left(1-\frac{\sqrt{H_2}}{H}\right)V\,d\Sigma\geq0.
		\end{align*}
		By other hand, Lemma \ref{lemma2} yields
		\begin{align*}
		\int_\Sigma\langle i\gamma(N)\psi,\varphi\rangle\,d\Sigma&=\sqrt{H_2}\int_\Sigma \langle\nabla V,N\rangle\,d\Sigma+\int_\Sigma V\,d\Sigma\\
		&=\int_{\Sigma}\sqrt{H_2}\langle\nabla V,N\rangle\,d\Sigma-\int_\Sigma\langle\nabla V,N\rangle H\,d\Sigma\\
		&=\int_\Sigma\left(\sqrt{H_2}-H\right)\langle\nabla V,N\rangle\,d\Sigma\leq0.
		\end{align*}
		Thus, we have
		$$ \int_\Sigma\langle i\gamma(N)\psi,\varphi\rangle\,d\Sigma=0.$$
		Thus, we should have equality in Lemma \ref{lemma2}, so $\Sigma$ is totally umbilical.
		
		Now, remains to think in the case where $\Sigma$ is compact but, is not the boundary of any compact domain. Define the height function $h\in\mathcal{C}^\infty(\Sigma)$ by setting $h=\pi_\mathbb{R}\circ f$, where $f:\Sigma\to\mathbb{R}\times_{\textrm{exp}}P$ is the isometric immersion ($h$ is nothing but the projection onto the real line). Since $\Sigma$ is compact there exists $p,q\in\Sigma$ such that $h$ attains his maximum and the minimum, respectively. If $h(q)=s_1$ and $h(p)=s_2$, then $\Sigma$ is contained in the region $\Omega_{s_1,s_2}$ bounded by the slices $\{s_1\}\times P$ and $\{s_2\}\times P$. As we have mentioned previously, the slices have constant principal curvatures equal 1, thus the second mean curvature of $\Sigma$ satisfies $H_2(p)\geq1$ and $H_2(q)\leq1$, hence $H_2\equiv1$.
		
		Now, choosing $u=e^h$ in (\ref{Lr}) and remembering that $\mathbb{R}\times_{\textrm{exp}}P$ is Einstein, by (\ref{div_curv}) we can obtain
		
		$$L_1(e^h)=n(n-1)e^h(H+\langle N,\partial_t\rangle H_2).$$
		But $H_2\equiv1$, $H\geq\sqrt{H_2}=1$ and Cauchy-Schwarz implies that $H+\langle N,\partial_t\rangle H_2\geq0$. 
		
		Hence, $L_1(e^h)\geq0$ on the compact manifold $\Sigma$. Thus, since in this case $L_1$ is elliptic , by the maximum principle applied to $L_1$ we conclude that $e^h$ is constant, and hence $h$ is constant, this implies that $\Sigma$ is a slice. 
		
		Thus, in both cases, $\Sigma$ is a totally umbilical hypersurface with constant $H_2$ curvature, this implies constant mean curvature. Now, applying lemma 4 of \cite{montiel}, where umbilical hypersurfaces with constant mean curvatures are classified. Thus, the result follows. 
		
	\end{proof}


\begin{thebibliography}{10}

\bibitem{aleksandrov}
A.~D. Aleksandrov.
\newblock Uniqueness theorems for surfaces in the large. {V}.
\newblock {\em Vestnik Leningrad. Univ.}, 13(19):5--8, 1958.

\bibitem{alias-lira}
L.~J. Al\'{i}as, J.~H.~S. de~Lira, and J.~M. Malacarne.
\newblock Constant higher-order mean curvature hypersurfaces in {R}iemannian
  spaces.
\newblock {\em J. Inst. Math. Jussieu}, 5(4):527--562, 2006.

\bibitem{bar}
C.~B\"ar.
\newblock Extrinsic bounds for eigenvalues of the {D}irac operator.
\newblock {\em Ann. Global Anal. Geom.}, 16(6):573--596, 1998.

\bibitem{baum}
H.~Baum.
\newblock Complete {R}iemannian manifolds with imaginary {K}illing spinors.
\newblock {\em Ann. Global Anal. Geom.}, 7(3):205--226, 1989.

\bibitem{brendle}
S.~Brendle.
\newblock Constant mean curvature surfaces in warped product manifolds.
\newblock {\em Publ. Math. Inst. Hautes \'Etudes Sci.}, 117:247--269, 2013.

\bibitem{cheng-yau}
S.~Y. Cheng and S.~T. Yau.
\newblock Hypersurfaces with constant scalar curvature.
\newblock {\em Math. Ann.}, 225(3):195--204, 1977.

\bibitem{elbert}
M.~F. Elbert.
\newblock Constant positive 2-mean curvature hypersurfaces.
\newblock {\em Illinois J. Math.}, 46(1):247--267, 2002.

\bibitem{garding}
L.~G{\aa}rding.
\newblock An inequality for hyperbolic polynomials.
\newblock {\em J. Math. Mech.}, 8:957--965, 1959.

\bibitem{HK}
E.~Heintze and H.~Karcher.
\newblock A general comparison theorem with applications to volume estimates
  for submanifolds.
\newblock {\em Ann. Sci. \'Ecole Norm. Sup. (4)}, 11(4):451--470, 1978.

\bibitem{hor}
O.~Hijazi, S.~Montiel, and S.~Raulot.
\newblock A holographic principle for the existence of imaginary {K}illing
  spinors.
\newblock {\em J. Geom. Phys.}, 91:12--28, 2015.

\bibitem{hss}
O.~Hijazi, S.~Montiel, and S.~Raulot.
\newblock On an inequality of brendle in the hyperbolic space.
\newblock {\em Comptes Rendus Mathematique}, pages~--, 2018.

\bibitem{hor2}
O.~Hijazi, S.~Montiel, and A.~Rold\'an.
\newblock Dirac operators on hypersurfaces of manifolds with negative scalar
  curvature.
\newblock {\em Ann. Global Anal. Geom.}, 23(3):247--264, 2003.

\bibitem{lawson}
H.~B. Lawson, Jr. and M.-L. Michelsohn.
\newblock {\em Spin geometry}, volume~38 of {\em Princeton Mathematical
  Series}.
\newblock Princeton University Press, Princeton, NJ, 1989.

\bibitem{minkowski}
H.~Minkowski.
\newblock Volumen und {O}berfl\"ache.
\newblock {\em Math. Ann.}, 57(4):447--495, 1903.

\bibitem{montiel}
S.~Montiel.
\newblock Unicity of constant mean curvature hypersurfaces in some {R}iemannian
  manifolds.
\newblock {\em Indiana Univ. Math. J.}, 48(2):711--748, 1999.

\bibitem{montiel2}
S.~Montiel and A.~Ros.
\newblock Compact hypersurfaces: the {A}lexandrov theorem for higher order mean
  curvatures.
\newblock In {\em Differential geometry}, volume~52 of {\em Pitman Monogr.
  Surveys Pure Appl. Math.}, pages 279--296. Longman Sci. Tech., Harlow, 1991.

\bibitem{morel}
B.~Morel.
\newblock Eigenvalue estimates for the {D}irac-{S}chr\"odinger operators.
\newblock {\em J. Geom. Phys.}, 38(1):1--18, 2001.

\bibitem{Qiu-Xia}
G.~Qiu and C.~Xia.
\newblock A generalization of reilly's formula and its applications to a new
  heintze--karcher type inequality.
\newblock {\em International Mathematics Research Notices},
  2015(17):7608--7619, 2014.

\bibitem{reilly}
R.~C. Reilly.
\newblock Applications of the {H}essian operator in a {R}iemannian manifold.
\newblock {\em Indiana Univ. Math. J.}, 26(3):459--472, 1977.

\bibitem{ros2}
A.~Ros.
\newblock Compact hypersurfaces with constant higher order mean curvatures.
\newblock {\em Rev. Mat. Iberoamericana}, 3(3-4):447--453, 1987.

\bibitem{ros}
A.~Ros.
\newblock Compact hypersurfaces with constant scalar curvature and a congruence
  theorem.
\newblock {\em J. Differential Geom.}, 27(2):215--223, 1988.
\newblock With an appendix by Nicholas J. Korevaar.

\bibitem{tashiro}
Y.~Tashiro.
\newblock Complete {R}iemannian manifolds and some vector fields.
\newblock {\em Trans. Amer. Math. Soc.}, 117:251--275, 1965.

\bibitem{wang}
M.~Y. Wang.
\newblock Parallel spinors and parallel forms.
\newblock {\em Ann. Global Anal. Geom.}, 7(1):59--68, 1989.

\bibitem{ww}
X.~Wang and Y.-K. Wang.
\newblock Brendle's inequality on static manifolds.
\newblock {\em J. Geom. Anal.}, 28(1):152--169, 2018.

\bibitem{yau}
S.~T. Yau, editor.
\newblock {\em Seminar on {D}ifferential {G}eometry}, volume 102 of {\em Annals
  of Mathematics Studies}.
\newblock Princeton University Press, Princeton, N.J.; University of Tokyo
  Press, Tokyo, 1982.
\newblock Papers presented at seminars held during the academic year
  1979--1980.

\end{thebibliography}

\end{document}